\documentclass{article}
\usepackage{fancyhdr}

\usepackage[british]{babel}

\usepackage[letterpaper,top=2cm,bottom=2cm,left=3cm,right=3cm,marginparwidth=1.75cm]{geometry}

\usepackage{amsmath}
\usepackage{amssymb}
\usepackage{amsthm}
\usepackage{graphicx}
\usepackage{xspace}
\usepackage[T1]{fontenc}
\usepackage[utf8]{inputenc}
\usepackage{float}
\usepackage{caption}
\usepackage{quiver}
\usepackage{yhmath}
\usepackage[nottoc,numbib]{tocbibind}
\usepackage{scalerel}
\usepackage{footmisc}
\usepackage[useregional]{datetime2}
\DTMlangsetup[en-GB]{showdayofmonth=false}

\usepackage[maxbibnames=99, style=alphabetic,backend=biber, url=false, doi=false, isbn=false, date=year]{biblatex}

\usepackage{filecontents}

\newtheorem{proposition}{Proposition}[section]
\newtheorem{theorem}[proposition]{Theorem}
\newtheorem{lemma}[proposition]{Lemma}
\newtheorem{example}[proposition]{Example}
\newtheorem{definition}[proposition]{Definition}
\newtheorem{remark}[proposition]{Remark}

\newtheorem{corollary}[proposition]{Corollary}

\makeatletter
\newcommand{\etale}{\'etal\@ifstar{\'e}{e\xspace}}
\makeatother
\newcommand\blfootnote[1]{%
  \begingroup
  \renewcommand\thefootnote{}\footnote{#1}%
  \addtocounter{footnote}{-1}%
  \endgroup
}

\newcommand{\Addresses}{{
  \bigskip
\noindent\textsc{Department of Mathematics, University of California, Berkeley,}\par\nopagebreak
\noindent Email: \texttt{tong.zhou@berkeley.edu}
  }}

\usepackage[colorlinks=true, allcolors=blue]{hyperref}

\newcommand{\CC}{\mathbf{C}}
\newcommand{\ZZ}{\mathbf{Z}}
\newcommand{\QQ}{\mathbf{Q}}
\newcommand{\NN}{\mathbf{N}}

\newcommand{\FF}{\mathbf{F}}
\newcommand{\PP}{\mathbf{P}}
\newcommand{\CF}{\mathcal{F}}
\newcommand{\CG}{\mathcal{G}}
\newcommand{\CM}{\mathcal{M}}
\newcommand{\CB}{\mathcal{B}}
\newcommand{\Gm}{\mathbf{G}_m}
\newcommand{\X}{\times}
\newcommand{\OX}{\otimes}
\newcommand{\BX}{\boxtimes}
\newcommand{\QQl}{\mathbf{Q}_\ell}

\newcommand{\QQbarl}{\overline{\mathbf{Q}}_\ell}
\newcommand{\FFbarl}{\overline{\mathbf{F}}_\ell}
\newcommand{\ZZbarl}{\overline{\mathbf{Z}}_\ell}
\newcommand{\CH}{\mathcal{H}}
\newcommand{\CL}{\mathcal{L}}
\newcommand{\CP}{\mathcal{P}}
\newcommand{\CK}{\mathcal{K}}
\newcommand{\AAA}{\mathbf{A}}
\newcommand{\DD}{\mathbb{D}}
\newcommand{\UF}{\underline{\mathcal{F}_1}}
\newcommand{\CA}{\mathcal{A}}
\newcommand{\RHom}{\mathcal{RH}om}

\title{The Fourier Transform and Characteristic Cycles of Monodromic $\ell$-adic Sheaves}
\author{Tong Zhou}
\date{}

\begin{document}
\maketitle
\begin{abstract}\blfootnote{February 2024}
Brylinski and Malgrange proved in 1986 that, for a monodromic algebraic D-module on a finite dimensional vector space over the complex numbers, its characteristic cycle is canonically identified with the characteristic cycle of its Fourier transform. We prove the exact analogue of this in the $\ell$-adic context. 
\end{abstract}


\section{Introduction}
Let $V=\mathrm{Spec}(\CC[x_1,x_2,...,x_d])$ be a finite dimensional vector space over $\CC$. Denote by $D(V)$ the triangulated category of bounded holonomic algebraic D-modules on $V$. $\CM\in D(V)$ is called \underline{monodromic} if the Euler vector field $\mathrm{eu}=\Sigma_i x_i\frac{\partial}{\partial x_i}$ acts locally finitely on each $\CH^i(\CM)$ (i.e., for any local section $s$, $\{\mathrm{eu}^n(s)\}_{n\in \NN}$ span a finite dimensional $\CC$-vector space). Denote by $D_{mon}(V)$ the full subcategory of monodromic D-modules. Let $F$ denote the Fourier transform of D-modules (c.f. \cite[\nopp 7.1]{katz_transformation_1985}). It is easy to see that being monodromic is preserved under the Fourier transform. We have:

\begin{theorem}[Brylinski-Malgrange, {\cite[\nopp 7.25]{brylinski_transformations_1986}}]\label{thm_bry_mal}
    1) If $\CM\in D_{mon}(V)$, then $CC(\CM)=CC(F\CM)$. 2) Further assume $\CM$ is regular, then so is $F\CM$.
\end{theorem}

Here $CC$ denotes the characteristic cycle, $V'$ denotes the dual of $V$, and $T^*V$ is implicitly canonically identified with $T^*V'$ via $T^*V=V\X V'\cong V'\X V\cong T^*V'$. Note that statement 1) in Theorem \ref{thm_bry_mal} as we stated is more general than Brylinski-Malgrange's original version, but in fact their proof works in this generality.\\    

The main theorem of this paper is the analogue of statement 1) for $\ell$-adic sheaves\footnote{Note that the analogue of statement 2) is false, i.e., being monodromic tame is not preserved under the Fourier transform. Example: let $k$ be an algebraically closed field of characteristic $p>0$, $V=\mathrm{Spec}(k[x,y,z]), Z=\{z^{p-1}y=x^p\}\hookrightarrow V$. Consider $\CF:=\underline{\Lambda}_Z$, which is evidently monodromic and tame. But, combining \cite[\nopp 9.13]{brylinski_transformations_1986} and the computation in \cite[\nopp 4.21]{zhou_stability_2023}, one sees that $F\CF$ is not tame.}. Let $V$ be a finite dimensional vector space over an algebraically closed field of characteristic $p>0$. Let $\Lambda$ be either a finite extension of $\FF_\ell$ (the finite case), or a finite extension of $\QQ_\ell$, or $\QQbarl$ (the rational case), for $\ell$ a prime different from $p$. Denote by $D(V)$ the triangulated category of bounded constructible $\Lambda$-\etale sheaves. We will prove:

\begin{theorem}[Corollary \ref{cor_main_text}]\label{thm_main}
    If $\CF\in D(V)$ is monodromic, then $CC(\CF)=CC(F\CF)$ and $SS(\CF)=SS(F\CF)$. 
\end{theorem}

Here $CC$ denotes the characteristic cycle, $SS$ denotes the singular support\footnote{We refer to \cites{beilinson_constructible_2016}{saito_characteristic_2017} for the theory of singular support and characteristic cycle for sheaves with finite coefficients, and to \cites{umezaki_characteristic_2020}{barrett_singular_2023} for the case of rational coefficients.}, and $F$ denotes the $\ell$-adic Fourier transform or its finite coefficient analogue (c.f. \cite{laumon_transformation_1987}). $\CF\in D(V)$ is called \underline{monodromic} if all $\CH^i(\CF)$ are \emph{tame} local systems on all $\Gm$-orbits. This is preserved under the Fourier transform (Proposition \ref{prop_basic_twist}.4).\\ 

In fact, we will prove the following theorem, which implies Theorem \ref{thm_main} by the additivity of characteristic cycles and singular supports with respect to irreducible constituents. We first introduce a terminology.

\begin{definition}[F-good]
    $\CF\in D(V)$ is \underline{F-good} if for each irreducible constituent $\CP$,\footnote{This means $\CP$ is an irreducible subquotient of some $^p\CH^i(\CF)$.} $CC(\CP)=CC(F\CP)$.
\end{definition}

\begin{theorem}[Theorem \ref{thm_main_text}]\label{thm_main'}
    Monodromic sheaves are F-good.
\end{theorem}

Our proof of Theorem \ref{thm_main'} consists of a precise realisation of the following intuition: a monodromic sheaf “decomposes” into a “projective component” and a “radial component” (the twist). The case where the twist is trivial can be proved utilising the relation between the Radon transform and the Fourier transform (c.f. \cite[\nopp 9.13]{brylinski_transformations_1986}) and the fact that characteristic cycles behave well under the Radon transform (\cite[\nopp 7.5]{saito_characteristic_2017}). The general case then follows, because the radial component is tame by monodromicity, and thus does not affect the characteristic cycle. Our original way of making the last sentence precise uses the notion of having the same wild ramification (see \cites{kato_wild_2018, kato_wild_2021} and references therein). Beilinson pointed out that the general case in fact follows formally from the trivial twist case by untwisting the sheaf after pulling back to $V\X\AAA^1$. This leads to a much simpler proof. Both proofs are presented.\\

In §\ref{sec_prelim}, we make a preliminary study on monodromic sheaves and F-good sheaves. In §\ref{sec_triv_twist}, we prove Theorem \ref{thm_main'} in the trivial twist case, and give a formula for the coefficient of $T^*_0V$ in $CC(\CF)$. In §\ref{sec_proof}, we prove Theorem \ref{thm_main'}. The Appendix reviews basic facts about characteristic cycles of sheaves with rational coefficients and the notion of having the same wild ramification.\\

In \cite{zhou_character_2024}, we will apply our results to give a microlocal characterisation of admissible (or character) sheaves on reductive Lie algebras in positive characteristic.

\subsection*{Conventions}
     We fix an algebraically closed field $k$ of characteristic $p>0$ and a prime $\ell\neq p$. A variety means a finite type reduced separated scheme over $k$. For a variety $X$, $D(X)$ denotes $D^b_c(X,\Lambda)$ (\cite[\nopp 1.1]{weil_II}). $\Lambda$ is either a finite extension of $\FF_\ell$, or a finite extension of $\QQ_\ell$, or $\QQbarl$. We refer to the former as the finite coefficients case, and the latter as the rational coefficients case. In all statements below, $\Lambda$ is understood to be either finite or rational unless otherwise specified. For $\CF\in D(X)$, by an irreducible constituent of $\CF$ we mean an irreducible subquotient of some $^p\CH^i(\CF)$.\\

     All derived categories are in the triangulated sense. All sheaf-theoretic functors are derived. A “sheaf” means an object of $D(X)$. A “local system” means an object of $D(X)$ whose cohomology sheaves are locally constant (if $\Lambda$ is finite) or lisse (if $\Lambda$ is rational) with finite type stalks. \\

     $V$ denotes a finite dimensional vector spaces over $k$, $V'$ denotes its dual, $\mathring{V}$ denotes $V-\{0\}$, $\PP(V)$ denotes the projectivisation of $\mathring{V}$, $q$ denotes the projection $\mathring{V}\rightarrow \PP(V)$.\\ 
    
     $\Gm$ acts on $V$ by scaling. For $n\geq 1$ in $\NN$, we call $\theta(n): \Gm\X V\rightarrow V, (\lambda,v)\mapsto \lambda^n v$ the $n$-twisted scaling action. We fix a non-trivial character $\psi: \ZZ/p \rightarrow \Lambda^\X$. Fourier transforms are denoted by $F$ and are with respect to this character unless otherwise specified. As we work over an algebraically closed field, we may ignore Tate twists.\\

     By a Kummer sheaf $\CK$, we mean the !-extension to $\AAA^1$ of a rank 1 local system in degree $-1$ on $\Gm$ corresponding to a non-trivial continuous character from the tame fundamental group $\pi^t_1(\Gm,1)$ to $\Lambda^\X$. We sometimes abuse notations and denote its restriction to $\Gm$ also by $\CK$. We denote by $\CK^{-1}$ the Kummer sheaf corresponding to the inverse character of that of $\CK$.

\subsection*{Acknowledgement}
I am grateful to David Nadler for many valuable discussions, and for suggesting various improvements to the draft of this paper. I would also like to thank Sasha Beilinson for discussions, especially for pointing out a simpler way to deduce the general case of our main theorem from the trivial twist case.

\section{Preliminaries on monodromic sheaves and F-good sheaves}\label{sec_prelim}
The setup is as in the Conventions. $\Lambda$ can be either finite or rational, unless otherwise specified. For completeness, we have included more materials in this section than are actually needed in the sequel. Recall:

\begin{definition}[monodromic sheaves, \cite{verdier_specialisation_1983}]\label{def_mon}
    A sheaf $\CF$ on $V$ is \underline{monodromic} if the restriction of all $\CH^i(\CF)$ to all $\Gm$-orbits are tame local systems.
\end{definition}

When $\Lambda$ is finite, we have the following crucial equivalent characterisation of monodromic sheaves.

\begin{proposition}[{\cite[\nopp 5.1]{verdier_specialisation_1983}}]\label{prop_verdier}
    Let $\Lambda$ be finite. Then, $\CF\in D(V)$ is monodromic if and only if $\exists$ $n>0$ in $\NN$ prime to $p$ such that there exists an isomorphism $\theta(n)^*\CF\tilde{\rightarrow}pr^*\CF$. Here $\theta(n): \Gm\X V\rightarrow V, (\lambda,v)\mapsto \lambda^n v$ is the n-twisted scaling action and $pr: \Gm\X V\rightarrow V$ is the projection.
\end{proposition}

\begin{proof}
    The “only if” direction is proved in \textit{loc. cit.} In \textit{loc. cit.}, it is not stated that $n$ can be chosen to be prime to $p$, but the proof in fact shows this.\\
    
    For the “if” direction, just observe that $\theta(n)^*\CF\tilde{\rightarrow}pr^*\CF$ implies $\theta(n)^*\CH^i(\CF)\tilde{\rightarrow}pr^*\CH^i(\CF)$. So for each $x\in V$, $(\theta(1)^*\CH^i(\CF))|_{\Gm\X\{x\}}$ is a sheaf concentrated in degree 0 and trivialised by the cover $\Gm\X\{x\}\rightarrow \Gm\X\{x\}, \lambda\mapsto\lambda^n, p\nmid n$. So $(\theta(1)^*\CH^i(\CF))|_{\Gm\X\{x\}}$, hence $\CH^i(\CF)|_{\Gm.x}$, is necessarily a tame local system. 

\end{proof}

The “if” direction is false for $\Lambda$ rational:

\begin{example}\label{ex_fin_mon}
    Let $\Lambda=\QQl$.
    
    1) Let $\CK$ be a Kummer sheaf whose corresponding representation of the tame fundamental group $\pi_1^t(\Gm,1)$ does not factor through a finite quotient, then $\CK$ cannot be trivialised by any finite cover (it has “infinite monodromy”), hence an $n$ as in the proposition does not exist.\\
    
    2) Consider the !-extension to $\AAA^1$ of the local system $\CL$ of rank 2 concentrated in degree $-1$ on $\Gm$ corresponding to the representation $\rho: \pi_1^t(\Gm,1)\rightarrow \mathbf{GL}_2(\Lambda), t\mapsto \left[ {\begin{array}{cccc}
   1 & 1 \\
   0 & 1 \\
  \end{array} } \right]$, where $t$ is a topological generator of $\pi_1^t(\Gm,1)$. This local system also has “infinite monodromy”, and an $n$ as in the proposition does not exist.
\end{example}

Let $\Lambda$ be finite or rational, and $\CF$ be a monodromic sheaf. If there exists an $n$ as in Proposition \ref{prop_verdier}, we say $\CF$ is \underline{finite monodromic} and refer to the (multiplicatively) smallest $n$ as the \underline{twist} of $\CF$, if furthermore the twist is 1, we say $\CF$ has \underline{trivial twist}. 



\begin{lemma}\label{lem_stab_mono}
    1) Being finite monodromic is preserved under taking irreducible constituents, $\OX$, and Verdier dual $\DD$.\\
    2) Being monodromic is preserved under taking cones, irreducible constituents, $\OX$, and Verdier dual $\DD$. In particular, a sheaf is monodromic if and only if its irreducible constituents are.
\end{lemma}

For $\Lambda$ rational, having finite monodromic irreducible constituents does not imply the sheaf itself is finite monodromic, as Example \ref{ex_fin_mon}.2 shows.

\begin{proof}
    1) Let $\CF, \CG$ be finite monodromic sheaves, and $n>0$ in $\NN$ prime to $p$ such that there exist isomorphisms $\theta(n)^*\CF\tilde{\rightarrow}pr^*\CF$, $\theta(n)^*\CG\tilde{\rightarrow}pr^*\CG$.\\

    $\theta(n)$ and $pr$ are smooth maps with connected geometric fibres, so $\theta(n)^*, pr^*$ are perverse t-exact and embeds $Perv(V)$ into $Perv(\Gm\X V)$ as a full subcategory closed under taking subquotients (\cite[\nopp 4.2.5]{BBDG}). We may thus take irreducible constituents on both sides of $\theta(n)^*\CF\tilde{\rightarrow}pr^*\CF$ and get the analogous isomorphisms for the irreducible constituents of $\CF$. So the irreducible constituents are also finite monodromic.\\
    
    The preservation under $\OX$ and $\DD$ is easily verified: $\theta(n)^*(\CF\OX\CG)\tilde{\rightarrow}pr^*(\CF\OX\CG)$, so $\CF\OX\CG$ is finite monodromic. $\theta(n)^*\DD\CF\cong\DD\theta(n)^!\CF\tilde{\rightarrow}\DD pr^!\CF\cong pr^*\DD\CF$, so $\DD\CF$ is finite monodromic.\\
    
    2) We first show the preservation under taking cones. Let $\CF\rightarrow\CG\rightarrow\CH\rightarrow$ be a distinguished triangle, with $\CF, \CG$ monodromic. The long exact sequence associated to $\CH^i$ easily implies that $\CH^i(\CH)$ sit inside exact sequences of the form $0\rightarrow coker_i\rightarrow\CH^i(\CH)\rightarrow ker_i\rightarrow0$. Restrict to any $\Gm$-orbit $\mathcal{O}$, $coker_i$ and $ker_i$ become cokernels and kernels of map between tame local systems, so are themselves tame local systems. $\CH^i(\CH)|_{\mathcal{O}}$ are thus also tame local systems.\\

    To show the preservation under taking irreducible constituents, because of the preservation under taking cones, we may do induction on the amplitude to reduce to the case of a monodromic sheaf $\CF$ concentrated in degree 0. Recall that for a sheaf $\CG$ concentrated in degree 0, being monodromic is equivalent to $\theta_\lambda^*\CG\cong\CG$, $\forall \lambda\in k^\X$, where $\theta_\lambda: V\rightarrow V$ is the map of multiplication by $\lambda$ (\cite[\nopp 3.2]{verdier_specialisation_1983}). So $\theta_\lambda^*\CF\cong\CF$. Since $\theta_\lambda^*$ restricts to an equivalence $Perv(V)\rightarrow Perv(V)$, we may take irreducible constituents on both sides and get $\theta_\lambda^*\CP\cong\CP$, for each irreducible constituent $\CP$ of $\CF$. Further take $\CH^i$, we get $\theta_\lambda^*\CH^i(\CP)\cong\CH^i(\CP)$. By \cite[\nopp 3.2]{verdier_specialisation_1983} again, $\CH^i(\CP)$ is monodromic.\\

    We now show the preservation under $\OX$. Let $\CF, \CG$ be monodromic sheaves. Because of the preservation under taking cones, we may do induction on the amplitude to reduce to the case where $\CF, \CG$ are concentrated in degree 0. Then, for any $\Gm$ orbit $\mathcal{O}$, $\CF|_\mathcal{O}$ and $\CG|_\mathcal{O}$ are tame local systems in degree 0. It follows that $(\CF|_\mathcal{O})\OX(\CG|_\mathcal{O})$ is a tame local system. So $\CH^i(\CF\OX\CG)|_\mathcal{O}=\CH^i((\CF|_\mathcal{O})\OX(\CG|_\mathcal{O}))$ is a tame local system, $\CF\OX\CG$ is monodromic.\\

    To show the preservation under $\DD$, because of the preservation under taking cones, we may reduce to the case of perverse irreducible monodromic sheaves. The finite coefficient case is dealt with in 1). For $\Lambda$ rational, we may further assume $\Lambda=\QQbarl$ because of the easily verified fact that, for $\Lambda$ rational, $\CF$ is monodromic if and only if $\CF\OX_\Lambda \QQbarl$ is. In this case, the statement follows from Proposition \ref{prop_basic_twist} items 1) and 3), and the compatibility of the Fourier transform and linear actions of algebraic groups (\cite[\nopp 1.2.3.4]{laumon_transformation_1987}).

\end{proof}

\begin{proposition}\label{prop_basic_twist}
    1) If $\CF\in D(V)$ is perverse, then $\CF$ is (finite monodromic) with trivial twist if and only if $\CF$ is $\Gm$-equivariant. This is preserved under the Fourier transform.\\
    2) Let $\CF\in D(V)$ be perverse irreducible finite monodromic with twist $n>1$. Assume $\Lambda$ contains a primitive $n$-th root of unity\footnote{This can always be achieved by adjoining a primitive $n$-th root of unity, see Remark \ref{rmk_twist_prim-to-l}.}. Then there exists a Kummer sheaf $\CK$ on $\Gm$ (unique up to isomorphism), trivialised by the power $n$ cover of $\Gm$, such that $\theta(1)^*\CF\cong \CK\BX\CF[-1]$. Furthermore, $\theta(1)^*F\CF\cong \CK^{-1}\BX F\CF[-1]$.\\
    3) Let $\Lambda=\QQbarl$. Let $\CF\in D(V)$ be perverse irreducible monodromic with twist $n>1$. Then there exists a Kummer sheaf $\CK$ on $\Gm$ (unique up to isomorphism), trivialised by the power $n$ cover of $\Gm$, such that $\theta(1)^*\CF\cong \CK\BX\CF[-1]$. Furthermore, $\theta(1)^*F\CF\cong \CK^{-1}\BX F\CF[-1]$.\\
    4) $\CF\in D(V)$ is monodromic (resp. finite monodromic) if and only if $F\CF$ is monodromic (resp. finite monodromic). In the finite monodromic case, they have the same twist.
\end{proposition}

Note that in situations 2) and 3), the restriction of $\CF$ to any $\Gm$-orbit not equal to \{0\} is of the form $\mathcal{C}\OX\CK$, for some \emph{constant} sheaf $\mathcal{C}$ (depending on the orbit). We also say that $\CK$ is the \underline{twist} of $\CF$.

\begin{proof}
    1) The first statement follows from the characterisation of perverse sheaves being equivariant under actions of connected algebraic groups (c.f. \cite[\nopp 6.2.17]{achar_perverse_2021}). The second statement follows from the compatibility of the Fourier transform and linear actions of algebraic groups.\\
    
    2) Being perverse irreducible, $\CF$ is of the form $j_{!*}\CL$ for some irreducible local system $\CL$ on some smooth irreducible locally closed conic subvariety $S\hookrightarrow \mathring{V}$.\footnote{Proof that $S$ can be chosen to be conic (note the proof only requires $\CF$ being perverse irreducible monodromic): assume $\CF$ is not a local system, let $D$ be its ramification divisor. We show $D$ is conic. Let $x$ be any closed point of $D$. Then, for some $i$, $\CH^i(\CF)$ is a local system near $x$. But $\CH^i(\CF)$ is monodromic, hence isomorphic to itself under the pullback by the $\lambda$-scaling, $\forall \lambda\in k^\X$ (\cite[\nopp 3.2]{verdier_specialisation_1983}), so $\CH^i(\CF)$ is not a local system near $\lambda.x, \forall \lambda\in k^\X$. This forces $D$ to be conic, as $D$ is a divisor.} The restriction of $\theta(n)^*\CF\tilde{\rightarrow}pr^*\CF$ to $\Gm\X S$ gives $\theta(n)|_{\Gm\X S}^*\CL\tilde{\rightarrow}pr|_{\Gm\X S}^*\CL$.\\

    \underline{Claim}: $\theta(1)|_{\Gm\X S}^*\CL$ is isomorphic to $\CK\BX\CL[-1]$ for some Kummer sheaf $\CK$ on $\Gm$ (necessarily unique up to isomorphism), trivialised by the power $n$ cover of $\Gm$.\\

    Accepting this claim, the conclusions follow: using the well-known characterisation of $j_{!*}$ (c.f. \cite[\nopp 3.3.4]{achar_perverse_2021}), it is easily seen that $\theta(1)^*\CF\cong\CK\BX pr^*\CF[-1]$. In fact, $\theta(1)^*\CF\cong\theta(1)^*(j_{!*}\CL)\cong j_{!*}(\theta(1)^*\CL)\cong j_{!*}(\CK\BX pr^*\CL[-1])$. Apply 3.3.4 in \textit{loc. cit.}, we get $\CK\BX j_{!*}pr^*\CL[-1]$ is the middle extension of $\CK\BX pr^*\CL[-1]$, hence isomorphic to $\theta(1)^*\CF$. The last statement follows from the compatibility of the Fourier transform and linear actions of algebraic groups.\\

    It remains to prove the claim. Denote $e(n): \Gm\rightarrow\Gm, \lambda\mapsto \lambda^n$. Consider $\CL':=(e(n)\X id)_*(e(n)\X id)^*\theta(1)|_{\Gm\X S}^*\CL$. Since $e(n)\X id$ is finite \etale, $\CL'$ is a local system concentrated in a single degree. Denote its corresponding $\pi_1(\Gm\X S)$ (we omit the base points in the notation from here on) representation by $\rho': \pi_1(\Gm\X S)\rightarrow \mathrm{Aut}_\Lambda(L')$. $\rho'$ factors through $\pi_1(\Gm\X S)\rightarrow\pi_1^t(\Gm)\X\pi_1(S)\rightarrow \ZZ/n\X\pi_1(S)$. The adjunction $id\rightarrow(e(n)\X id)_*(e(n)\X id)^*$ realises $\theta(1)|_{\Gm\X S}^*\CL$ as a sub-local-system of $\CL'$. Its corresponding representation $\rho: \pi_1(\Gm\X S)\rightarrow \mathrm{Aut}_\Lambda(L_1)$ thus also factors through $\ZZ/n\X\pi_1(S)$. Since $\theta(1)|_{\Gm\X S}^*\CL$ is irreducible, $\rho$ is irreducible. By our assumption on $\Lambda$, we can apply Lemma \ref{lem_rep_decomp} case 1) and get $L_1\cong K\BX L$ as $\ZZ/n\X\pi_1(S)$ representations, for some 1-dimensional representation $K$ of $\ZZ/n$. Consequently $\theta(1)|_{\Gm\X S}^*\CL\cong \CK\BX\CG$, for some Kummer sheaf $\CK$ and some sheaf $\CG$. Looking at the restriction of $\theta(1)|_{\Gm\X S}^*\CL$ to $1\X S\hookrightarrow \Gm\X S$, we see $\CG$ is necessarily isomorphic to $\CL[-1]$. $\CK$ is clearly trivialised by the power $n$ cover of $\Gm$.\\

    3) The argument is similar to 2), we indicate the differences. Consider $\theta(1)|_{\Gm\X S}^*\CL$ as above. Fix a torsion free integral model for $\theta(1)|_{\Gm\X S}^*\CL$. For each of its reductions mod $\ell^r$, the corresponding $\pi_1(\Gm\X S)$-representation (over $\ZZ/\ell^r$) factors through $\ZZ/m\X\pi_1(S)$ for varying $m$. Take the limit over $r$, we get that the (continuous) $\pi_1(\Gm\X S)$-representation over $\QQbarl$ corresponding to $\theta(1)|_{\Gm\X S}^*\CL$ factors through $\pi^t_1(\Gm)\X\pi_1(S)$. It is necessarily irreducible. Apply Lemma \ref{lem_rep_decomp} case 2) (easily modified to take continuity into account), we get an external product decomposition of $\theta(1)|_{\Gm\X S}^*\CL$. The rest is similar.\\

    4) The statements concerning the finite monodromic case follow from the compatibility of the Fourier transform and linear actions of algebraic groups, and the fact that being finite monodromic implies $a^*\CF\cong\CF$, where $a$ is the antipodal map. The statement concerning the monodromic ($\Lambda$ rational) case follows from 3) above, Lemma \ref{lem_stab_mono}.1, and the easily verified fact that, for $\Lambda$ rational, $\CF$ is monodromic if and only if $\CF\OX_\Lambda \QQbarl$ is.    

\end{proof}

\begin{lemma}\label{lem_rep_decomp}
    Let $H$ be an abelian group, $G$ be any group, $\Lambda$ be a field. Assume either 1) $H=\ZZ/n$ for $n>1$ in $\NN$, and $\Lambda$ contains a primitive $n$-th root of unity, or 2) $\Lambda$ is algebraically closed. Then, for any finite dimensional irreducible $\Lambda$-representation $M$ of $H\X G$, there exist irreducible $\Lambda$-representations $K$ of $H$ and $L$ of $G$ and an isomorphism $K\BX L\cong M$ as representations of $H\X G$. Note, necessarily, $\mathrm{dim}_\Lambda(K)=1$, $\mathrm{dim}_\Lambda(L)=\mathrm{dim}_\Lambda(M)$.
\end{lemma}

Here $\BX$ denotes the external tensor product of group representations. It is also denoted by $\OX$ in the literature.

\begin{proof}
    In case 1), all finite dimensional representations of $\ZZ/n$ are semisimple (note that $n$ is necessarily invertible in $\Lambda$), and irreducible ones are 1-dimensional. View $M$ as a representation of $\ZZ/n=\ZZ/n\X\{1\}$, it decomposes as $M=\oplus_i K_i^{\oplus r_i}$, for some 1-dimensional representations $K_i$ of $\ZZ/n$, and $r_i>1$. Since the $G=\{1\}\X G$ action commutes with the $\ZZ/n$ action, each $K_i^{\oplus r_i}$ is a sub-representation of $\ZZ/n\X G$. By the irreducibility of $M$, there is only one of them. Denote it by $M=K^r$. View $M$ as a representation of $G$, and denote it by $L$, then clearly $M\cong K\BX L$ as representations of $\ZZ/n\X G$.\\

    In case 2), since $H\X\{1\}$ and $\{1\}\X G$ commute, and $H$ is abelian, the homomorphism $H\X\{1\}\rightarrow End_\Lambda M$ lands in $End_{H\X G-rep}M$. By Schur's Lemma, $End_{H\X G-rep}M\cong\Lambda$. We see that each element of $H\X\{1\}$ must act through scaling. Denote the corresponding 1-dimensional representation of $H$ by $K$, and $M$ regarded as a $G$-representation (clearly irreducible) by $L$. Then $M\cong K\BX L$ as representations of $H\X G$.

\end{proof}

\begin{remark}\label{rmk_twist_prim-to-l}
    We own the following observation to Beilinson: for $\Lambda$ finite, and $\CF\in D(V)$ a perverse irreducible monodromic sheaf, the twist $n$ is always prime to $\ell$. Proof: using the same notations as in the fourth paragraph of the proof of Proposition \ref{prop_basic_twist}.2, we claim the representation $\rho: \ZZ/n\X\pi_1(S)\rightarrow \mathrm{Aut}_\Lambda L_1$ must factor through $((\ZZ/n)/\{\ell-torsion\})\X\pi_1(S)$. Since it follows easily from the definition of the twist that $n$ is the smallest positive integer for which a factorisation $\pi_1(\Gm\X S)\rightarrow\ZZ/n\X\pi_1(S)\rightarrow \mathrm{Aut}_\Lambda L_1$ exists, $n$ must be prime to $\ell$. To see the claim, note that $\ZZ/n\X\{1\}$ is in the centre of $\ZZ/n\X\pi_1(S)$, so it lands in $End_{\ZZ/n\X\pi_1(S)-rep}L_1$. As $\rho$ is irreducible, the latter is a division algebra over $\Lambda$ by Schur's Lemma. If $m\in\ZZ/n\X\{1\}$ is $\ell$-torsion, say $\ell^rm=0$, then $\rho(\ell^rm)=\rho(m)^{\ell^r}=id$, so $\rho(m)^{\ell^r}-id=(\rho(m)-id)^{\ell^r}=0$, so $\rho(m)=id$.
\end{remark}

\begin{proposition}\label{prop_loc_decomp}
    Let $\CF\in D(V)$ be perverse irreducible monodromic, with non-trivial twist $\CK$. Assume either 1) $\CF$ is finite monodromic with twist $n$ and $\Lambda$ contains a primitive $n$-th root of unity, or 2) $\Lambda=\QQbarl$. Let ($W,\sigma$) be the data of an open conic subvariety $W$ of $\mathring{V}:=V-\{0\}$ together with a section $\sigma$ of the projection $q: \mathring{V}\rightarrow \PP(V)$ (restricted to $W$). Then, $\CF|_W\cong \underline{\CF}_\sigma\BX\CK$ for some perverse irreducible sheaf $\underline{\CF}_\sigma$ on $\PP(V)$ (unique up to isomorphism).
\end{proposition}

Here, $\sigma$ determines an isomorphism $W\cong_\sigma\underline{W}\X\Gm$, and the $\BX$ is with respect to this isomorphism. We emphasise that $\underline{\CF}_\sigma$ depends on $\sigma$.

\begin{proof}
    $\CF$ is of the form $j_{!*}\CL$ for some irreducible local system $\CL$ on some smooth irreducible locally closed conic subvariety $S\hookrightarrow \mathring{V}$, $\CF|_W\cong j_{!*}(\CL|_{W\cap S})$. Consider the sheaf $(\CL|_{W\cap S})\OX pr_2^*\CK^{-1}$, where $pr_2: W\cong_\sigma\underline{W}\X\Gm\rightarrow\Gm$ is the second projection. By the comment after the statement of Proposition \ref{prop_basic_twist}, $(\CL|_{W\cap S})\OX pr_2^*\CK^{-1}$ is a local system concentrated in a single degree and constant on each closed fibre of the projection $pr_1: \underline{W\cap S}\X\Gm\rightarrow\underline{W\cap S}$. Apply Lemma \ref{lem_product_decomp}, we get $(\CL|_{W\cap S})\OX pr_2^*\CK^{-1}\cong pr_1^*\CL'[2]$ for some (perverse) local system $\CL'$ on $\underline{W\cap S}$. (In fact $\CL'$ must be isomorphic to $((\CL|_{W\cap S})\OX pr_2^*\CK^{-1})|_{\underline{W\cap S}\X \{1\}}[-2]$.) So $\CL|_{W\cap S}\cong \CL'\BX\CK$. Reasoning as in the third paragraph in proof of Proposition \ref{prop_basic_twist}.2, the well-known characterisation of $j_{!*}$ implies $\CF|_W\cong j_{!*}(\CL|_{W\cap S})\cong j_{!*}(\CL'\BX\CK)\cong (j_{!*}\CL')\BX\CK=:\underline{\CF}_\sigma\BX\CK$.
    
\end{proof}

\begin{lemma}\label{lem_product_decomp}
    Let $f: X\rightarrow Y$ be a smooth morphism between varieties of relative dimension $d$, with geometrically connected fibres. If $\CF$ is a sheaf on $X$ concentrated in degree 0, such that for each closed point $y\in Y$, there exists an isomorphism $\CF|_{X_y}\cong\underline{\Lambda}^r$, for some $r_i\in\NN$ independent of $y$. Then the canonical map $\CF\rightarrow f^*\CH^{2d}(f_!\CF)$ is an isomorphism. If $\CF$ is perverse, then $\CH^{2d}(f_!\CF)$ is the unique (up to isomorphism) sheaf on $Y$ (necessarily perverse and concentrated in degree 0) whose pullback is isomorphic to $\CF$.
\end{lemma}

Here the map $\CF\rightarrow f^*\CH^{2d}(f_!\CF)$ is obtained by taking $\CH^0$ of the adjunction map $\CF\rightarrow f^!f_!\CF$.

\begin{proof}\footnote{This proof follows the suggestion of Will Sawin in https://mathoverflow.net/questions/225468}
    It suffices to show $\CF\rightarrow f^*\CH^{2d}(f_!\CF)$ is an isomorphism for each closed point $x\in X$. This, in turn, is implied by $\CF|_{X_y}\tilde{\rightarrow} (f^*\CH^{2d}(f_!\CF))|_{X_y}\cong f^*\CH^{2d}(f_!(\CF|_{X_y}))$ for each closed point $y\in Y$, where the last isomorphism is from proper base change. Using $\CF|_{X_y}\cong\underline{\Lambda}^r$, the question reduces to showing $\underline{\Lambda}_{X_y}\tilde{\rightarrow}p^*\CH^{2d}(p_!\underline{\Lambda}_{X_y})$, which is clear (here we use the connectedness of $X_y$). The assertion when $\CF$ is perverse is a direct consequence of the fact that $f^*$ induces a fully faithful embedding of $Perv(Y)$ into $Perv(X)$ (\cite[\nopp 4.2.5]{BBDG}, here we use again the geometrically-connectedness of fibres).
    
\end{proof}

\begin{remark}
    Proposition \ref{prop_loc_decomp} can fail without the assumptions on $\Lambda$. In fact, if $\CL|_{W\cap S}$ corresponds to an irreducible representation of $\pi_1(\underline{W\cap S})\X\ZZ/n$ which cannot be written as an external tensor product (which can exist without the assumptions on $\Lambda$), then $\CL|_{W\cap S}$ cannot be written as an external tensor product. 
\end{remark}

We now turn to F-good sheaves. Recall:

\begin{definition}[F-good sheaves]
    $\CF\in D(V)$ is \underline{F-good} if for each irreducible constituent $\CP$, $CC(\CP)=CC(F\CP)$.
\end{definition}

\begin{remark}
    F-good sheaves are not necessarily monodromic. For example, let $\CL$ be a local system concentrated in degree 0 on $\Gm$, $\CF$ be the !-extension of $\CL[1]$ to $\AAA^1$. If $\CL$ is purely of slope $<1$ at $\infty$, then $\CF$ is F-good, as one can verify using Laumon’s local Fourier transforms (c.f. \cite[\nopp 2.3.1, 2.4.3]{laumon_transformation_1987}). But such an $\CF$ is not monodromic if $\CL$ is not tame.
\end{remark}

\begin{lemma}\label{lem_stab_Fgood}
    1) Being F-good is preserved under taking cones, taking irreducible constituents, and Verdier dual $\DD$.\\
    2) Let $f: W\rightarrow V$ be a linear injection (resp. surjection) between finite dimensional vector spaces, $\CF$ (resp. $\CG$) be an F-good sheaf on $W$ (resp. $V$). Then $f_*\CF$ (resp. $f^*\CG$) is F-good.
\end{lemma}

\begin{proof}
    1) That being F-good is preserved under taking irreducible constituents is clear. If $\CF\rightarrow\CG\rightarrow\CH\rightarrow$ is a distinguished triangle, the long exact sequence associated to $^p\CH^i$ easily implies that irreducible constituents of $\CG$ is a subset of the union of irreducible constituents of $\CF$ and $\CH$. So the F-goodness of $\CF$ and $\CH$ implies the F-goodness of $\CG$. Finally, let $\CF$ be F-good, we show $\DD\CF$ is F-good: as $\DD$ is an anti-equivalence preserving $Perv(V)$, it suffices to prove $CC(\DD\CF)=CC(F\DD\CF)$ for $\CF$ perverse irreducible. Apply the formula $F\DD\cong a^*\DD F$ to $\CF$ (where $a$ is the “multiplication by $-1$” on $V$), using the monodromicity of $\DD F\CF$, we get $F\DD\CF\cong\DD F\CF$. So $CC(F\DD\CF)=CC(\DD F\CF)$. By $CC\DD=CC$ (see \cite[\nopp 5.13.4]{saito_characteristic_2017} for $\Lambda$ finite, and Proposition \ref{prop_CC_radon_rat} for $\Lambda$ rational), and use the assumption that $\CF$ is F-good, we get $CC(\DD\CF)=CC(\CF)=CC(F\CF)=CC(\DD F\CF)=CC(F\DD\CF)$.\\

    2) This follows directly from the compatibility of the Fourier transform with linear maps (\cite[\nopp 1.2.2.4]{laumon_transformation_1987} and its dual version), and the behaviour of $CC$ under closed immersions and smooth pullbacks (see \cite[\nopp 5.13.2, 5.17]{saito_characteristic_2017} for $\Lambda$ finite, the rational case follows easily from the finite case and the definition of $CC$, reviewed in the Appendix).   
    
\end{proof}

\section{The case of the trivial twist}\label{sec_triv_twist}
As above, $\Lambda$ can be either finite or rational. We prove the special case of Theorem \ref{thm_main'} where the sheaf has trivial twist, which will be the basis for the proof of the general case. In the terminology of the intuition mentioned in the Introduction, we deal with the “projective components” in this section.

\begin{proposition}\label{prop_triv_Fgood}
    If $\CF\in D(V)$ is perverse irreducible with trivial twist, then $\CF$ is F-good.
\end{proposition}

Recall that, for $\CF$ perverse, having trivial twist is equivalent to being $\Gm$-equivariant, and this is preserved under the Fourier transform (Proposition \ref{prop_basic_twist}.1). We fix some notations. Let $\pi: \tilde{V}\rightarrow V$ be the blowup of $V$ at 0, $\tilde{q}: \tilde{V}\rightarrow \PP(V)$ be the natural projection, $j$ be the inclusion $\mathring{V}\subseteq V$. We use the same letters with “ $'$ ” to denote the corresponding maps on the dual side.

\begin{proof}
    We first prove $CC(\CF)=CC(F\CF)$ away from the 0-sections and 0-fibres (i.e., with the components supported on $V\X 0$ and $0\X V'$ removed). Consider the following diagram, where each sequence is a distinguished triangle:
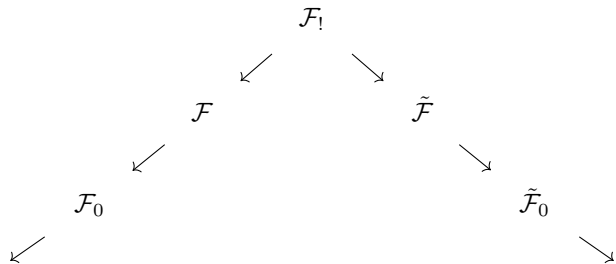
\[\begin{tikzcd}
	&&& {\mathcal{F}_!} \\
	&& {\mathcal{F}} && {\tilde{\mathcal{F}}} \\
	& {\mathcal{F}_0} &&&& {\tilde{\mathcal{F}}_0} \\
	{} &&&&&& {}
	\arrow[shorten <=8pt, shorten >=8pt, from=1-4, to=2-3]
	\arrow[shorten <=8pt, shorten >=8pt, from=2-3, to=3-2]
	\arrow[shorten <=8pt, shorten >=8pt, from=1-4, to=2-5]
	\arrow[shorten <=8pt, shorten >=8pt, from=2-5, to=3-6]
	\arrow[shorten <=8pt, shorten >=8pt, from=3-6, to=4-7]
	\arrow[shorten <=8pt, shorten >=8pt, from=3-2, to=4-1]
\end{tikzcd}\]
Here $\CF_!:=j_!(\mathcal{F}|_{\mathring{V}})$, $\tilde{\CF}:=\pi_*\tilde{q}^*\mathcal{F}$, and $\CF_0$ (resp. $\tilde{\CF}_0$) is the stalk of $\CF$ (resp. $\tilde{\CF}$) at 0, viewed as skyscraper sheaves. As $\CF$ is $\Gm$-equivariant, $\CF|_{\mathring{V}}$ descends to some $\CG$ on $\PP(V)$. By the compatibility of the Radon transform and the Fourier transform (\cite[\nopp 9.13]{brylinski_transformations_1986}), $(F\tilde{\CF})|_{\mathring{V}'}\cong q'^*R\CG$, where $R$ is the Radon transform on $\PP(V)$. By the smooth pullback formula for $CC$ and the compatibility of $CC$ with the Radon transform (see \cite[\nopp 7.5]{saito_characteristic_2017} for $\Lambda$ finite, and Proposition \ref{prop_CC_radon_rat}.3 for $\Lambda$ rational) this easily implies $CC(F\tilde{\CF})=CC(\CF)$ away from the 0-sections and 0-fibres. By the above diagram and its Fourier dual, $CC(\CF)=CC(\tilde{\CF})-CC(\tilde{\CF}_0)+CC(\CF_!)$, $CC(F\CF)=CC(F\tilde{\CF})-CC(F\tilde{\CF}_0)+CC(F\CF_!)$. Since the last two terms in each equality are supported on the 0-sections and 0-fibres, we conclude that $CC(\CF)=CC(F\CF)$ away from the 0-sections and 0-fibres.\\

To prove the full equality $CC(\CF)=CC(F\CF)$, consider $i: V\hookrightarrow V\X\AAA^1, v\mapsto (v,0)$ and its dual $p': (V\X\AAA^1)'\rightarrow V'$. $i_*\CF$ is still perverse irreducible with trivial twist, so, by the above paragraph, $CC(i_*\CF)=CC(Fi_*\CF)=CC(p'^*F\CF[1])$ away from the 0-sections and 0-fibres (of $V\X\AAA^1$). The 0-section in $CC(\CF)$ corresponds to the component $\{(v,adt), v\in V, a\in k\}$ in $CC(i_*\CF)$ ($t$ is the linear coordinate on $\AAA^1$), and the 0-fibre in $CC(F\CF)$ corresponds to the component $\{(0,\xi), \xi\in V\cong T^*_{0'}V'\}$ in $CC(p'^*F\CF[1])$. They are away from the 0-sections and 0-fibres (of $V\X\AAA^1$), hence equal. This proves the 0-section in $CC(\CF)$ equals the 0-fibre in $CC(F\CF)$. Apply the Fourier inversion, we get the 0-section in $CC(F\CF)$ equals the 0-fibre in $CC(\CF)$. Hence the full equality $CC(\CF)=CC(F\CF)$.

\end{proof}

\begin{corollary}
    If $\CF\in D(V)$ is such that all its irreducible constituents have trivial twists (equivalently, $\Gm$-equivariant), then $\CF$ is F-good.
\end{corollary}

\begin{corollary}\label{cor_triv_Fgood}
    If $\CF\in D(V)$ is such that $\CF|_{\mathring{V}}\cong q^*\underline{\CF}$ for some $\underline{\CF}\in D(\PP(V))$, then $\CF$ is F-good.
\end{corollary}

\begin{proof}
    By the corollary above, it suffices to show all irreducible constituents of $\CF$ are $\Gm$-equivariant. Let $\CP$ be an irreducible constituent. If $\CP$ is supported at $\{0\}$, this is clear. If not, then $\CP$ is of the form $j_{!*}\CL$ for $\CL$ some perverse irreducible local system on some smooth irreducible subvariety in $\mathring{V}$. So $\CP|_{\mathring{V}}$ is still perverse irreducible. Since the restriction to $\mathring{V}$ is perverse t-exact, $\CP|_{\mathring{V}}$ is an irreducible constituent of $\CF|_{\mathring{V}}\cong q^*\underline{\CF}$. Since $q^*$ is also perverse t-exact and induces a fully faithful embedding of $Perv(\PP(V))$ into $Perv(\mathring{V})$ closed under taking subquotients (\cite[\nopp 4.2.5]{BBDG}), it is easily seen that irreducible constituents of $q^*\underline{\CF}$ are exactly $q^*$ of irreducible constituents of $\underline{\CF}$. So $\CP|_{\mathring{V}}\cong q^*\underline{\CP}$ for some perverse irreducible $\underline{\CP}$ on $\PP(V)$. So $\CP|_{\mathring{V}}$, hence $\CP$, is $\Gm$-equivariant. 
    
\end{proof}

\begin{remark}
    It follows from the proof that the irreducible constituents of such an $\CF$ have trivial twists, so $\CF$ is necessarily monodromic (Lemma \ref{lem_stab_mono}). Note that, for $\Lambda$ finite, $\CF$ needs not have trivial twist; for $\Lambda$ rational, $\CF$ needs not be of finite monodromy. The Fourier transform of $j_{!*}\CL$ for $\CL$ as in Example \ref{ex_fin_mon}.2 and $j: \Gm\rightarrow\AAA^1$ gives such an example ($j_{!*}\CL$ is in fact the maximal extension of $\underline{\Lambda}_{\Gm}[1]$).
\end{remark}

For a monodromic sheaf $\CF\in V$, its singular support is $\Gm$-stable\footnote{This can be seen, e.g., using Proposition \ref{prop_loc_decomp}.}. The 0-fibre $SS(\CF)\cap T^*_0V$ is either $T^*_0V$ or $\overline{SS(\CF)|_{\mathring{V}}}\cap T^*_0V$ (closure taken in $T^*V$). As an application of the above, we record a formula, applicable to perverse sheaves with trivial twists, which allows us to tell which case happens. This will not be used in the sequel.

\begin{proposition}
    Let $\CF\in D(V)$ be such that  $\CF|_{\mathring{V}}\cong q^*\underline{\CF}[1]$ for some sheaf $\underline{\CF}$ on $\PP(V)$. Then, the coefficient of $T^*_0V$ in $CC(\CF)$ equals $rk_0\CF-\chi(\PP(V),\underline{\CF}[1])+\chi(H,\underline{\CF}[1])$, where $\chi$ denotes the Euler characteristic, and $H$ is a general hyperplane on $\PP(V)$.\footnote{More precisely: there exists an open dense $U\subseteq \PP(V')$, such that $\chi(\tilde{H},\underline{\CF}[1])$, as a function of hyperplanes $\tilde{H}$ (parametrised by closed points of $\PP(V')$), is constant. $\chi(H,\underline{\CF}[1])$ is defined to be this constant.} If $\CF$ is also perverse, then $SS(\CF)\cap T^*_0V=T^*_0V$ if and only if $rk_0\CF-\chi(\PP(V),\underline{\CF}[1])+\chi(H,\underline{\CF}[1])\neq 0$. 
\end{proposition}

\begin{proof}
    It suffices to prove the first statement, as the second statement follows from the first and the effectivity of characteristic cycles of perverse sheaves. Denote $\mathrm{dim}V$ by $d$. It follows from Corollary \ref{cor_triv_Fgood} that the coefficient of $T^*_0V$ in $CC(\CF)$ equals $(-1)^d.rk(F\CF)$, where $rk$ denotes the generic rank. Using the same diagram and notations at the beginning of the proof of Proposition \ref{prop_triv_Fgood}, we get $rk(F\CF)=rk(F\CF_0)+rk(F\tilde{\CF})-rk(F\tilde{\CF}_0)$. Compute:\\
$rk(F\CF_0)=(-1)^d.rk_0(\CF)$;\\
$rk(F\tilde{\CF}_0)=(-1)^d.rk(\tilde{\CF}_0)=(-1)^d.\chi(\PP(V),\underline{\CF}[1])$;\\
$rk(F\tilde{\CF})=rk(q^*R\underline{\CF}[1])=rk(R\underline{\CF}[1])=(-1)^{d-2}.\chi(H,\underline{\CF}[1])$,\\
where in the last line we have used the compatibility of the Radon transform and the Fourier transform. The statement easily follows. 

\end{proof}

\section{Proof of the main theorem}\label{sec_proof}
As above, $\Lambda$ can be either finite or rational.
\begin{theorem}\label{thm_main_text}
    Monodromic sheaves are F-good.
\end{theorem}

We present two proofs. The first one follows Beilinson's key idea of untwisting the sheaf after pulling back to $V\X\AAA^1$, reducing the general case to the trivial twist case.

\begin{proof}
    Since irreducible constituents of a monodromic sheaf are monodromic (Lemma \ref{lem_stab_mono}), it suffices to prove the claim for perverse irreducible monodromic sheaves. Further, since being monodromic is clearly preserved under a finite extension of the coefficient field, and characteristic cycles do not change under this extension (c.f. the discussion after Definition \ref{def_CC_rat}), we may assume that $\Lambda$ contains a primitive $n$-th root of unity\footnote{In the extended coefficient field $\Lambda'$, we will use the character $\psi': \ZZ/p\rightarrow\Lambda\hookrightarrow\Lambda'$ to define the Fourier transform, where the first arrow is the character $\psi$ for $\Lambda$. This ensures $(F_\psi\CF)\OX_\Lambda\Lambda'\cong F_{\psi'}(\CF\OX_\Lambda\Lambda')$.}, where $n$ is the twist of the sheaf in consideration. In the rational case, we may further assume $\Lambda=\QQbarl$.\\

    Let $\CF\in D(V)$ be perverse irreducible monodromic. We want to show $CC(\CF)=CC(F\CF)$. If $\CF$ has trivial twist, then it is F-good by Proposition \ref{prop_triv_Fgood}. Assume $\CF$ has non-trivial twist $\CK$. Consider $\CF\BX\CK$ on $V\X\AAA^1$. We claim that $\CF\BX\CK^{-1}$ satisfies $(\CF\BX\CK^{-1})|_{(V\X\AAA^1)-\{0\}}\cong q^*(\underline{\CF\BX\CK^{-1}})$ for some $\underline{\CF\BX\CK^{-1}}\in D(\PP(V\X\AAA^1))$, where $q: (V\X\AAA^1)-\{0\}\rightarrow\PP(V\X\AAA^1)$ is the projection. Accepting this claim, Corollary \ref{cor_triv_Fgood} implies that $CC(\CF\BX\CK^{-1})=CC(F(\CF\BX\CK^{-1}))$. As, in general, $CC(\CF_1\BX \CF_2)=CC(\CF_1)\BX CC(\CF_2)$ (for $\Lambda$ finite, see \cite[\nopp 2.2]{saito_characteristic_ext_2017}; for $\Lambda$ rational, this is verified in Proposition \ref{prop_CC_radon_rat}).4 and $F$ commutes with $\BX$ (c.f. \cite[\nopp 1.2.2.7]{laumon_transformation_1987}), it easily follows that $CC(\CF)=CC(F\CF)$.\\

    It remains to show the claim. $\CF$ is of the form $j_{!*}\CL$ for some irreducible local system $\CL$ on some smooth irreducible locally closed conic subvariety $S\hookrightarrow \mathring{V}$. So $\CF\BX\CK^{-1}\cong (j_{!*}\CL)\BX\CK^{-1}\cong j_{!*}(\CL\BX\CK^{-1})$. By our construction, $\CL\BX\CK^{-1}$ is a local system concentrated in a single degree, which is constant when restricted to each $\Gm$ orbit in $S\X\AAA^1$ (note that $\CK^{-1}$ is 0 at $\{0\}\in\AAA^1$). By Lemma \ref{lem_product_decomp}, $\CL\BX\CK^{-1}\cong q^*(\underline{\CL\BX\CK^{-1}})$, for some $\underline{\CL\BX\CK^{-1}}\in D(\PP(S\X\AAA^1))$. Then $\CF\BX\CK^{-1}\cong j_{!*}q^*(\underline{\CL\BX\CK^{-1}})$. Its restriction to $(V\X\AAA^1)-\{0\}$ is isomorphic to $q^*j_{!*}(\underline{\CL\BX\CK^{-1}})$.

\end{proof}

We now present our original proof, which uses the local decomposition (Proposition \ref{prop_loc_decomp}) and the notion of having the same wild ramification to reduce to the trivial twist case.

\begin{proof}
    As explained at the beginning of the previous proof, it suffices to consider perverse irreducible monodromic sheaves, and we may assume, in the finite case, that $\Lambda$ contains a primitive $n$-th root of unity where $n$ is the twist of the sheaf in consideration, or, in the rational case, that $\Lambda=\QQbarl$ .\\
    
    We do induction on $d=\mathrm{dim}V$. For $d=1$, there are three types of perverse irreducible monodromic sheaves: i) the rank 1 skyscraper at $\{0\}$, ii) the rank 1 constant sheaf in degree $-1$ on $V$, iii) (!-extension of) Kummer sheaves $\{\CK\}$. Their Fourier transforms are easy to compute: i') the rank 1 constant sheaf in degree $-1$ on $V'$, ii') the rank 1 skyscraper at $\{0'\}$, iii') (!-extension of) Kummer sheaves $\{\CK^{-1}\}$ (\cite[\nopp 1.4.3.2]{laumon_transformation_1987}). In each case, F-goodness can be directly verified.\\

    Now consider the case $d>1$. Let $\CF\in D(V)$ be a perverse irreducible monodromic sheaf. If $\CF$ has trivial twist, then it is F-good by Proposition \ref{prop_triv_Fgood}. Assume $\CF$ has non-trivial twist $\CK$ (recall, by our convention, $\CK$ is in degree $-1$). Fix linear coordinates $(x_1,x_2,...,x_d)$ (i.e. an isomorphism $V\cong\AAA^d=\mathrm{Spec}(k[x_1,x_2,...,x_d])$), this induces coordinates $[x_1:x_2:...:x_d]$ on $\PP(V)$. Let $D_1=\{x_1=0\}$, $U_1$ its complement in $V$. The projection $q: \mathring{V}\rightarrow\PP(V)$ maps $U_1$ to $\underline{U_1}=\{x_1\neq 0\}\subseteq\PP(V)$. Fix the section $\sigma_1: \underline{U_1}\rightarrow U_1, [x_1:x_2,...,x_d]\mapsto (1,\frac{x_2}{x_1},...,\frac{x_d}{x_1})$.\\
    
    Apply Proposition \ref{prop_loc_decomp} to $(U_1,\sigma_1)$, we get a decomposition $\CF|_{U_1}\cong\underline{\CF}_{\sigma_1}\BX\CK$. We denote the !-extension of $\underline{\CF}_{\sigma_1}$ to $\PP(V)$ by $\UF$. Then $j_!(\CF|_{U_1})\cong j_!(\underline{\CF}_{\sigma_1}\BX\CK)\cong (j_!q^*\UF)\OX pr_1^*\CK$, where $j$ denotes the inclusions into $V$ (we use the same $j$ for the inclusion from $U_1$ as well as $\mathring{V}$), and $pr: V\rightarrow \AAA^1$ denotes the projection to the first coordinate. The last isomorphism follows from the observation that the map $U_1\cong_{\sigma_1}\underline{U_1}\X\Gm$ followed by the projection to $\Gm$ coincides with the map $pr_1$ (restricted to $U_1$).\\

    We have the distinguished triangle: $j_!(\CF|_{U_1})\rightarrow\CF\rightarrow i_{1*}(\CF|_{D_1})\rightarrow$, where $i_{1*}$ is the inclusion of $D_1$ to $V$. As $\CF|_{D_1}$ is clearly monodromic, it is F-good by the induction hypothesis. Using the compatibility of the Fourier transform with linear maps between vector spaces, it is easily seen that $i_{1*}(\CF|_{D_1})$ is F-good. As being F-good is stable under taking cones (Lemma \ref{lem_stab_Fgood}), it suffices to show $j_!(\CF|_{U_1})$ is F-good, i.e., to show $CC(Fj_!(\CF|_{U_1}))=CC(j_!(\CF|_{U_1}))$. In the following, we assume $j_!(\CF|_{U_1})$ is nonzero (with non-trivial twist $\CK$).\\

    We compute: $Fj_!(\CF|_{U_1})=F((j_!q^*\UF)\OX pr_1^*\CK)=(Fj_!q^*\UF)\ast F(pr_1^*\CK)[d]$, where $-\ast-$ denotes the convolution: let $s$ be the sum map: $V'\X V'\rightarrow V', (v_1,v_2)\mapsto v_1+v_2$, then $-\ast-: D(V')\X D(V')\rightarrow D(V'), (\CG_1,\CG_2)\mapsto s_!(\CG_1\BX\CG_2)$. Further compute: $F(pr_1^*\CK)=i'_{1!}F\CK[1-d]=i'_{1!}\CK^{-1}[1-d]$, where $i'_1$ is the inclusion of the $x_1'$-axis into $V'$ (we use the dual coordinates on $V'$).\\
    
    \underline{Claim}: $(Fj_!q^*\UF)\BX (i'_{1!}\CK^{-1})$ has the same wild ramification (swr) as $(Fj_!q^*\UF)\BX (i'_{1!}\underline{\Lambda}_{\Gm}[1])$.\\

    Accepting the claim, then by Theorem \ref{thm_swr_kato} and Theorem \ref{thm_swr}, $(Fj_!q^*\UF)\ast F(pr_1^*\CK)[d]$ has the swr as $(Fj_!q^*\UF)\ast (i'_{1!}\underline{\Lambda}_{\Gm}[1])[1]$, and they have the same characteristic cycle. Now, $i'_{1!}\underline{\Lambda}_{\Gm}[1]=FFi'_{1!}\underline{\Lambda}_{\Gm}[1]=Fpr_1^*\CH[d-1]$, where $\CH:=Fi'_{1!}\underline{\Lambda}_{\Gm}[1]$ is a sheaf on $\AAA^1$ whose restriction to $\Gm$ is constant and concentrated in degree $-1$. So $(Fj_!q^*\UF)\ast (i'_{1!}\underline{\Lambda}_{\Gm}[1])=F((j_!q^*\UF)\OX pr_1^*\CH)[-1]$. Note that $(j_!q^*\UF)\OX pr_1^*\CH$ is isomorphic to $j_!q^*\UF[1]$, which is F-good by Corollary \ref{cor_triv_Fgood}. Put these together, we get the equalities $CC(Fj_!(\CF|_{U_1}))=CC(F((j_!q^*\UF)\OX pr_1^*\CH)[1-d])=CC(j_!q^*\UF[1])$. Recall that we want to show this equals $CC(j_!(\CF|_{U_1}))$.\\

    \underline{Claim}: $j_!(\CF|_{U_1})=(j_!q^*\UF)\OX pr_1^*\CK$ has the swr as $(j_!q^*\UF)\OX pr_1^*\underline{\Lambda}_{\Gm,!}[1]=j_!q^*\UF[1]$, where $\underline{\Lambda}_{\Gm,!}$ denotes the !-extended to $\AAA^1$ of the constant sheaf on $\Gm$.\\

    Accepting this claim and combining it with the previous equalities, we get the desired equality: $CC(Fj_!(\CF|_{U_1}))=CC(j_!q^*\UF[1])=CC(j_!(\CF|_{U_1}))$.\\

    It remains to prove the two claims. We prove the second claim, the first is completely analogous. We first consider the $\Lambda$ finite case. Denote $j_!q^*\UF$ by $\CA$, $\CK$ by $\CB$. Let $\overline{C}$ be a smooth proper curve, $g: C\subseteq \overline{C}$ an open dense, $f: C\rightarrow X$ a map, $s\in \overline{C}$ a closed point. We want to show $a_s(g_!f^*\CA[1])=a_s(g_!f^*(\CA\OX pr_1^*\CB))$. Recall $a_s=rk_{\overline{\eta}_s}+sw_{\overline{\eta}_s}-rk_s$. Clearly, the ranks are the same for $g_!f^*\CA[1]$ and $g_!f^*(\CA\OX pr_1^*\CB)$ (note $\CA$ is 0 along $D_1$). For the Swan conductors, observe $(\CA\OX pr_1^*\CB)_{\overline{\eta}_s}=\CA_{\overline{\eta}_s}\OX (pr_1^*\CB)_{\overline{\eta}_s}$, $(pr_1^*\CB)_{\overline{\eta}_s}$ is 0 (if $\eta_s$ is mapped to 0 by $pr_1\circ f$) or tame of rank 1 concentrated in degree $-1$ (otherwise). In both cases, $sw_{\overline{\eta}_s}(\CA\OX pr_1^*\CB)=sw_{\overline{\eta}_s}(\CA\OX pr_1^*\underline{\Lambda}_{\Gm,!}[1])=sw_{\overline{\eta}_s}(\CA[1])$. This completes the proof for the $\Lambda$ finite case. For the $\Lambda=\QQbarl$ case, it suffices to make the following changes to this paragraph: $\CA$ denotes the reduction of any integral model of $j_!q^*\UF$, and $\CB$ denotes the reduction of any torsion free integral model of $\CK$. Note, $\CB$ is then a rank 1 local system concentrated in degree $-1$ trivialised by a power $n$ cover of $\Gm$, $p\nmid n$, hence a Kummer sheaf.

\end{proof}

\begin{corollary}\label{cor_main_text}
    If $\CF\in D(V)$ is monodromic, then $CC(\CF)=CC(F\CF)$ and $SS(\CF)=SS(F\CF)$.
\end{corollary}

\begin{proof}
    The characteristic cycle (resp. singular support) of a sheaf is the sum (resp. union) of the characteristic cycles (resp. singular supports) of its irreducible constituents. The Fourier transform preserves the irreducible constituents. So it suffices to prove $CC(\CF)=CC(F\CF)$ and $SS(\CF)=SS(F\CF)$ for $\CF$ perverse irreducible monodromic. The first equality follows from the theorem above. The support of the characteristic cycle of a perverse sheaf equals its singular support (for $\Lambda$ finite, this is \cite[\nopp 5.17]{saito_characteristic_2017}; for $\Lambda$ rational, this is verified in Proposition \ref{prop_CC_radon_rat}.2), the second equality follows.

\end{proof}

\section{Appendix: review of the characteristic cycle and the notion of having the same wild ramification}
\subsection{The characteristic cycle of a sheaf with rational coefficient} We refer to \cite[\nopp §5]{umezaki_characteristic_2020} and \cite{zheng_six_2015} for details. Let $X$ be a variety over $k$. Denote the Grothendieck group of constructible $\FFbarl$ (resp. $\ZZbarl$, resp. $\QQbarl$)-sheaves on $X$ by $K(X, \FFbarl)$ (resp. $K(X, \ZZbarl)$, resp. $K(X, \QQbarl)$). There are natural group homomorphisms: 
\[\begin{tikzcd}
	{K(X, \overline{\mathbf{F}}_\ell)} & {K(X, \overline{\mathbf{Z}}_\ell)} & {K(X, \overline{\mathbf{Q}}_\ell)}
	\arrow["{i^*}", shift left=2, from=1-1, to=1-2]
	\arrow["{i_*}", shift left, from=1-2, to=1-1]
	\arrow["{j^*}", from=1-2, to=1-3]
\end{tikzcd}\]
where $i^*$, $i_*$, and $j^*$ are induced by the reduction, restricting scalars, and  tensoring to $\QQbarl$, respectively. It is known that $i_*=0$, $i^*$ is surjective, and $j^*$ is an isomorphism. Define the \underline{decomposition homomorphism} $d: K(X, \QQbarl)\rightarrow K(X, \FFbarl)$ as $i^*\circ (j^*)^{-1}$. 

\begin{definition}[$CC$ for rational coefficients, {\cite[\nopp 5.3.2]{umezaki_characteristic_2020}}]\label{def_CC_rat}
    Let $\Lambda$ be rational. For $\CF\in D(X)$, $CC(\CF):=CC(d[\CF\OX_\Lambda \QQbarl])$. Here “[  ]” denotes the class in $K(X, \QQbarl)$. 
\end{definition}

We will drop “[  ]” and “$-\OX_\Lambda\QQbarl$” from the notation if there is no risk of confusion. Here by $CC(d[\CF\OX_\Lambda \QQbarl])$ we mean the characteristic cycle of any representative for $\CF\OX_\Lambda \QQbarl$ which is defined over some finite extension of $\FF_\ell$. This is well-defined because $CC$ is additive and does not change under coefficient field extensions (which can be seen, for example, using the Milnor formula and the fact that Swan conductors do not change under coefficient field extensions).\\

Concretely, $CC(\CF)$ can be computed as follows: let $Q$ be a large enough finite extension of $\QQl$ on which $\CF$ is defined. Denote by $Z$ its ring of integers, and by $F$ the residue field. Choose any integral model $\CF_0$ for $\CF$ (i.e. a $Z$-sheaf $\CF_0$ such that $\CF_0\OX_Z Q\cong \CF$). Let $\overline{\CF}_0=\CF_0\OX_Z F$ be the reduction. Then $[\overline{\CF}_0]=d[\CF\OX_E\QQbarl]$, and $CC(\CF)=CC(\overline{\CF}_0)$. \\

As the operations $f^*, f_*, f^!, f_!, \OX$ and $\mathcal{RH}om$ are exact functors between triangulated categories, they induce the corresponding operations on the Grothendieck groups, denoted by the same letters. The decomposition homomorphism commutes with all these operations. For $f^*, f_*, f^!, f_!$, this is stated in \cite[\nopp 5.2.7]{umezaki_characteristic_2020}, for $\OX$ and $\mathcal{RH}om$, this is verified in the following. 

\begin{lemma}\label{lem_d_OXRHom}
    Let $\Lambda$ be rational. Let $\CF, \CG$ be sheaves on a variety. Then $d(\CF\OX\CG)=(d\CF)\OX (d\CG), d(\RHom(\CF,\CG))=\RHom(d\CF,d\CG)$.
\end{lemma}

\begin{proof}
    Suppose $\CF, \CG$ are defined over a finite extension $Q$ of $\QQl$, denote by $Z$ (resp. $F$) the ring of integers (resp. residue field) of $Q$. Let $\CF_0, \CG_0$ be any integral models for $\CF, \CG$, denote their reductions by $\overline{\CF}_0, \overline{\CG}_0$.\\

    Essentially by the definition of $-\OX_Q-$ (\cite[\nopp 6.1]{zheng_six_2015}), $\CF_0\OX_Z\CG_0$ is an integral model for $\CF\OX_Q\CG$. So $d(\CF\OX\CG)=(\CF_0\OX_Z\CG_0)\OX_Z F=\overline{\CF}_0\OX_F\overline{\CG}_0=(d\CF)\OX(d\CG)$, where in the second equality we used \cite[\nopp 5.3]{zheng_six_2015}.\\

    Essentially by the definition of $\RHom_Q(-,-)$, $\RHom_Z(\CF_0,\CG_0)$ is an integral model for $\RHom_Q(\CF,\CG)$. So, $d(\RHom(\CF,\CG))=\RHom_Z(\CF_0,\CG_0)\OX_Z F=\RHom_F(\CF\OX_Z F, \CG\OX_Z F)=\RHom(d\CF,d\CG)$, where in the second equality we used \cite[\nopp 5.7]{zheng_six_2015}.

\end{proof}

One can thus transport results of characteristic cycles proved in the finite coefficient case to the rational coefficient case. Here are a few that we need but not explicitly stated in \cite{umezaki_characteristic_2020}.

\begin{proposition}\label{prop_CC_radon_rat}
    Let $\Lambda$ be rational.\\
    1) If $\CF$ is a sheaf on a smooth variety, then $CC\DD(\CF)=CC(\CF)$.\\
    2) Let $\CF$ be a sheaf on a smooth variety. Then $CC(\CF)$ is supported on $SS(\CF)$. If $\CF$ is perverse, nonzero, then the coefficients in $CC(\CF)$ is positive on each irreducible component of $SS(\CF)$. In particular, the support of $CC(\CF)$ equals $SS(\CF)$.\\
    3) If $\CG$ is a sheaf on a projective space $\PP$, then $CC(R\CG)=LCC(\CG)$, where $R$ is the Radon transform and $L$ is the Legendre transform (as defined above 7.5 in \cite{saito_characteristic_2017}).\\
    4) Let $X, Y$ be smooth varieties, $\CF_1, \CF_2$ be sheaves on $X, Y$, respectively. Then $CC(\CF_1\BX \CF_2)=CC(\CF_1)\BX CC(\CF_2)$ (see \cite[§2]{saito_characteristic_ext_2017} for the meaning of the notation).
\end{proposition}

\begin{proof}
    1) We want to show $CC(\DD\CF)=CC(\CF)$. By definition, $CC(\DD\CF)=CC(d\DD\CF), CC(\CF)=CC(d\CF)$. By the corresponding result for finite coefficients (\cite[\nopp 5.13.4]{saito_characteristic_2017}), it suffices to show $d\circ\DD=\DD\circ d$, which is immediate from the commutativity of $d$ with $\RHom$.\\

    2) Let $\CF_0$ be an integral model for $\CF$, and $\overline{\CF}_0$ its reduction mod $\ell$, such that $SS(\CF)=SS(\overline{\CF}_0)$ (which exists, by \cite[\nopp 1.5 (v)]{barrett_singular_2023}). By definition, $CC(\CF)=CC(\overline{\CF}_0)$. The first claim follows. The second claim follows from the Milnor formula (\cite[\nopp 5.3.3]{umezaki_characteristic_2020}) and the well-known fact that the vanishing cycle shifted by $-1$ is perverse t-exact (c.f. \cite[\nopp 4.6]{illusie_autour_1994}).\\
    
    3) and 4) follow from the commutativity of $d$ with $f^*$, $f_*$, and $\OX$, and the corresponding results for finite coefficients (\cites[\nopp 7.12]{saito_characteristic_2017}[\nopp 2.2]{saito_characteristic_ext_2017}).

\end{proof}

\subsection{The notion of having the same wild ramification} In situations relevant to us, this notion is equivalent to having universally the same conductors (\cite[\nopp 6.11]{kato_wild_2021}). We will only review (and use) the latter, as it is easier to state and verify (in our situation). We refer to \cites{kato_wild_2018, kato_wild_2021} and references therein for details.

\begin{definition}[universally the same conductors for finite coefficients, {\cite[\nopp 2.5]{kato_wild_2018}}]
    Let $\Lambda$ be finite. Let $X$ be a variety over $k$. We say $\CF, \CF'\in D(X)$ have \underline{universally the same conductors} (usc), if for all smooth proper curve $\overline{C}$, all open dense $j: C\subseteq \overline{C}$, all map $f: C\rightarrow X$, all closed point $s\in \overline{C}$, we have $a_s(j_!f^*\CF)=a_s(j_!f^*\CF')$, where $a_s:=rk_{\overline{\eta}_s}+sw_{\overline{\eta}_s}-rk_s$ is the Artin conductor at $s$.
\end{definition}

\begin{theorem}[{\cite[\nopp 4.6.ii, 4.7]{kato_wild_2018}}]\label{thm_swr_kato}
    Let $\Lambda$ be finite. Let $f: X\rightarrow Y$ be a map between varieties.\\
    1) If $\CF, \CF'\in D(X)$ have usc, then $f_!\CF, f_!\CF'\in D(Y)$ have usc.\\
    2) Assume $X$ is smooth. If $\CF, \CF'\in D(X)$ have usc, then $CC(\CF)=CC(\CF')$.
\end{theorem}

Note that as $a_s$, $j_!$ and $f^*$ are additive, having usc descends to the Grothendieck group. This suggests that we can transport this notion to rational coefficients and get the analogue of the above theorem.

\begin{definition}[same wild ramification for rational coefficients]
    Let $\Lambda$ be rational. Let $X$ be a variety over $k$. We say $\CF, \CF'\in D(X)$ have the \underline{same wild ramification} (swr), or have \underline{universally the same conductors} (usc), if $d(\CF), d(\CF')$ do.
\end{definition}

\begin{theorem}\label{thm_swr}
    Let $\Lambda$ be rational, $f: X\rightarrow Y$ be a map between varieties.\\
    1) If $\CF, \CF'\in D(X)$ have the swr, then $f_!\CF, f_!\CF'\in D(Y)$ have the swr.\\
    2) Assume $X$ is smooth. If $\CF, \CF'\in D(X)$ have the swr, then $CC(\CF)=CC(\CF')$.
\end{theorem}

\begin{proof}
    1) This follows from the corresponding statement for $\Lambda$ finite and the fact that the decomposition homomorphism commutes with $f_!$.\\
    
    2) This follows from the corresponding statement for $\Lambda$ finite and the definition of $CC$ for $\Lambda$ rational.

\end{proof}

\newpage
\printbibliography

\Addresses

\end{document}